\documentclass[11pt]{article}
\usepackage{amssymb, amsmath, amsthm, amsfonts, enumerate}
\usepackage{amsmath,setspace,scalefnt}
\usepackage[dvipsnames,svgnames,table]{xcolor}
\usepackage{graphicx,tikz,caption,subcaption}
\usetikzlibrary{patterns}
\usetikzlibrary{shapes}
\usetikzlibrary{decorations.pathreplacing}
\usetikzlibrary{decorations.pathmorphing}
\usetikzlibrary{positioning}
\usepackage{fullpage}
\usepackage{hyperref}
\usepackage[capitalise]{cleveref}

\usepackage{tikz}
\tikzstyle{aNode} = [circle, fill = black]
\tikzstyle{bNode} = [circle,draw = black, thick]

\newcommand{\ppoints}[1]{%
\begin{tikzpicture}[inner sep = 0.7pt, #1]%
\node (1) at (0,-2) [aNode]{};
\node (3) at (1.5,-2) [aNode]{};
\node (2) at (0.75,-1) [aNode]{};
\end{tikzpicture}%
}
\def\points{\ppoints{scale=0.08}}

\newtheorem{theorem}{Theorem}

\newtheorem{lemma}[theorem]{Lemma}

\newtheorem*{definition}{Definition}

\newtheorem{claim}[theorem]{Claim}

\newcommand\pa{{\pi_{\points}^{\operatorname{pal}}}}
\newcommand\pu{{\pi_{\points}}}

\newcommand\cp{{\mathcal{P}}}
\newcommand\cc{{\mathcal{C}}}
\newcommand\ca{{\mathcal{A}}}

\newcommand\ex{\ensuremath{\mathrm{ex}}}

\title{The uniform Tur\'an density of large stars}
\author{Ander Lamaison \thanks{Extremal Combinatorics and Probability Group (ECOPRO), Institute for Basic Science (IBS), Daejeon, South Korea. Supported by IBS-R029-C4. Email: ander@ibs.re.kr}
\and
Zhuo Wu\thanks{Mathematics Institute and DIMAP, University of Warwick, Coventry, UK, and Extremal Combinatorics and Probability Group (ECOPRO),  Institute for Basic Science (IBS), Daejeon, South Korea. Supported by the Warwick Mathematics Institute Centre for Doctoral Training and funding from University of Warwick and the Institute for Basic Science (IBS-R029-C4). Email: Zhuo.wu@warwick.ac.uk}
}

\begin{document}

\maketitle

\begin{abstract}
We asymptotically resolve the the uniform Tur\'an density problem for the large stars. In particular, we show that the uniform Tur\'an density of the $k$-star $S_k$ is $\frac{k^2-5k+7}{(k-1)^2}$ for $k\ge 48$, matching a lower construction by Reiher, R\"odl and Schacht.
\end{abstract}

\begin{section}{Introduction}

Tur\'an problems constitute one of the central areas of study in extremal combinatorics. Given an $r$-uniform hypergraph (or $r$-graph for short) $F$ and a positive integer $n$, the extremal number $\ex(n,F)$ is the maximum number of edges in an $r$-graph $F$ not containing $F$ as a subgraph. The Tur\'an density of $F$ is the limit 
\[\pi(F):=\lim\limits_{n\rightarrow\infty}\frac{\ex(n,F)}{\binom{n}{r}}.\]

While Tur\'an problems on graphs are well-understood~\cite{Man07, Tur41, ErdS46}, much less is known about Tur\'an problems on hypergraphs. For example, the exact value of the Tur\'an density of $K_k^{(r)}$, the complete $r$-graph on $k$ vertices, is not known for any $k>r\geq 3$. In fact, the Tur\'an density of $K_4^{(3)-}$, which is obtained by removing an edge from $K_4^{(3)}$, is currently unknown.

In the conjectured extremal constructions for most hypergraph Tur\'an problems, the set of edges is distributed very unevenly among the vertices. For instance, there is a family of $K_4^{(3)}$-free $3$-graphs whose density tends to $5/9$ (conjectured to be the value of $\pi(K_4^{(3)})$). In each hypergraph in the family, the vertex set can be partitioned into three independent sets. This motivated Erd\H{o}s and S\'os~\cite{ErdS82, Erd90} to propose studying Tur\'an density problems with a condition on the edge distribution.

\begin{definition}
A $3$-graph $H$ is said to be \emph{$(d,\varepsilon,\points)$-dense}
if any subset $S\subseteq V(H)$ contains at least $d\binom{|S|}{3}-\varepsilon|V(H)|^3$ edges. 
The \emph{uniform Tur\'an density} $\pu(F)$ of a $3$-graph $F$
is defined as the infimum of the values of $d$, 
for which there exists $\varepsilon>0$ and $N$ such that 
every $(d,\varepsilon,\points)$-dense hypergraph
on at least $N$ vertices
contains $F$ as a subgraph.
\end{definition}

In the last decade, a number of results about the uniform Tur\'an density of $3$-graphs have been published. The first non-zero value of $\pu$ to be computed was $\pu(K_4^{(3)-})=1/4$, solving a conjecture of Erd\H{o}s and S\'os. This was first proved by Glebov, Kr\'al' and Volec \cite{GleKV16} using flag algebras, and later by Reiher, R\"odl and Schacht~\cite{ReiRS18a} using the hypergraph regularity method. The latter set of authors also characterized in~\cite{ReiRS18} the $3$-graphs $F$ with $\pu(F)=0$, and as a consequence, deduced that there does not exist a $3$-graph $F$ with $\pu(F)\in(0,1/27)$.

Buci\'c, Cooper, Kr\'al', Mohr and Munha-Correia~\cite{BucCKMM23} determined the uniform Tur\'an densities of all tight cycles of length at least 5. They showed that $\pu(C_\ell^{(3)})$ equals $0$ if $\ell$ is divisible by 3, and $4/27$ otherwise. Additionally, $3$-graphs with uniform Tur\'an density equal to $1/27$ and $8/27$ have also been found~\cite{GarKL24,GIKKL24}. Further results can be found in a survey by Reiher~\cite{Rei20}.

One simple hypergraph family whose uniform Tur\'an density has been studied is stars. Let $S_k$ be the $3$-graph on $k+1$ vertices $u,v_1, v_2, \dots, v_k$, containing the edges $uv_iv_j$ for all $1\leq i<j\leq k$. We call this graph the $k$-\emph{star}. For example, $S_3$ is just the graph $K_4^{(3)-}$. Despite the simplicity of their structure, attempts to extend the idea used to determine the uniform Tur\'an density of $K_4^{(3)-}$ to any larger star have been unsuccessful. The current best bounds were given by Reiher, R\"odl and Schacht~\cite{ReiRS18a}, who showed that \[\frac{k^2-5k+7}{(k-1)^2}\leq\pu(S_k)\leq \left(\frac{k-2}{k-1}\right)^2.\]

In this paper we show that the lower bound is sharp for all large stars, and thereby asymptotically resolve the the uniform Tur\'an density problem for the large stars.

\begin{theorem}\label{thm:main}

$\pu(S_k)=\frac{k^2-5k+7}{(k-1)^2}$ for all $k\geq 48$.

\end{theorem}

The key concept in the proof of Theorem~\ref{thm:main} is the notion of palette, introduced by Reiher~\cite{Rei20} as a generalization of a construction of R\"odl~\cite{Rod86}.

\begin{definition}\label{def:palette}
A \emph{palette} $\cp$ is a pair $(\mathcal{C},\mathcal{A})$, 
where $\mathcal{C}$ is a finite set (whose elements we call \emph{colors}) 
and a set of (ordered) triples of colors $\mathcal{A}\subseteq \mathcal{C}^3$, 
which we call the \emph{admissible triples}.
The density of $\cp$ is $d(\cp):=|\mathcal{A}|/|\mathcal{C}|^3$.

We say that a $3$-graph $F$ \emph{admits} a palette $\cp$
if there exists an order $\preceq$ on $V(F)$
and a function $\varphi:\binom{V(F)}{2}\rightarrow \mathcal{C}$
such that for every edge $uvw\in E(F)$ with $u\prec v\prec w$
we have $\left(\varphi(uv), \varphi(uw), \varphi(vw)\right)\in \mathcal{A}$.
\end{definition}

Palettes can be used to give lower bounds on the uniform Tur\'an density of $3$-graphs. Indeed, if the $3$-graph $F$ does not admit the palette $\cp$, then one can use an analogous of R\"odl's construction to obtain a $(d(\cp),o(1),\points)$-dense family of $F$-free hypergraphs, meaning that $\pu(F)\geq d(\cp)$. For this reason, it is worth considering the largest density of a palette which is not admitted by $F$.

\begin{definition}
 The \emph{palette Tur\'an density} of a 3-graph $F$ is
 
 \[\pa(F):=\sup\{d(\cp):\cp\text{ palette, }F\text{ does not admit }\cp\}.\]
\end{definition}

As discussed above, for every $3$-graph $F$ we have $\pu(F)\geq \pa(F)$. In~\cite{Lam24}, the first author showed that equality always holds. This result will be the main tool in the proof of Theorem~\ref{thm:main}.

\begin{theorem}[\cite{Lam24}]\label{thm:palette}
For every $3$-graph $F$, we have $\pu(F)=\pa(F)$.
\end{theorem}

\paragraph{Idea.} The main technical step we develop to prove Theorem~\ref{thm:main} is transforming the hypergraph condition to a digraph condition (see  \cref{cla:trans}). We rely on a property of the structure of $S_k$: each edge $uvw$ contains a pair of vertices which is not contained in any other edge.  That means that, when we consider the palette condition $\left(\varphi(uv), \varphi(uw), \varphi(vw)\right)\in \mathcal{A}$, one of the three entries can be selected freely without affecting any other edge. For this reason it is enough to consider, for each $1\leq i<j\leq 3$, the pairs of colors $(a,b)$ for which there exists a triple $(c_1, c_2, c_3)\in\ca$ such that $c_i=a$ and $c_j=b$. 

%Then, we use two important estimation step among the proof: the first one is \cref{lem:inverse}, which is a variant of Caro-Wei Theorem in digraph, and give an degree inequality over transitive $K_k$-free digraphs. The second one is \cref{lem:incexc}, which using inclusion-exclusion principle to bound the palette density by the number of $(i,j)$-good pairs.

\paragraph{Notation.} For any palette $\cp=(\cc, \ca)$, let $\ca_{a}^{1}, \ca_{a}^{2},\ca_{a}^{3}$ be the set of the triples in $\ca$ with the first, second, third entry equal to $a$, respectively. The palette $\cp'_{a}$, i.e, the palette obtained by removing the color $a$ from $\cp$ is 
\[\cp'_{a}=(\cc\setminus\{a\},  (\cc\setminus\{a\})^3\cap\ca). \]

\end{section}

\begin{section}{Proof of Theorem~\ref{thm:main}}

By Theorem~\ref{thm:palette}, it is sufficient to show that $\pa(S_k)=\frac{k^2-5k+7}{(k-1)^2}$. The lower bound was first shown to hold in~\cite{ReiRS18a}, here we will give a short proof for completeness.

Consider the palette $\cp=(\mathcal{C},\mathcal{A})$ with $\mathcal{C}=\{0, \cdots, k-2\}$,
\[\mathcal{A}=\{(x,y,z)| \;x\neq y, \; y\neq z, \; z\not\equiv x+1\pmod {k-1}\}. \]
It is easy to calculate that $\cp$ has density $\frac{k^2-5k+7}{(k-1)^2}$. Now we prove that $\cp$ does not admit $S_k$, which shows that $\pa(S_k)\ge \frac{k^2-5k+7}{(k-1)^2}$.

For contradiction, assume that there exist an order $\preceq$ on $V(F)$ and a function $\varphi:\binom{V(F)}{2}\rightarrow \mathcal{C}$
such that for every edge $uvw\in E(F)$ with $u\prec v\prec w$
we have $\left(\varphi(uv), \varphi(uw), \varphi(vw)\right)\in \mathcal{A}$. By symmetry, assume that the order $\preceq$ on $V(F)$ is $v_1\preceq\cdots\preceq v_t \preceq u\preceq v_{t+1}\cdots\preceq v_k$. Consider the following $k$ elements:
\[\varphi(uv_1)+1,\cdots, \varphi(uv_t)+1, \;\; \varphi(uv_{t+1}),  \cdots, \varphi(uv_k).\]

By the pigeonhole principle, at least two of them are the same modulo $k-1$. Now we have 3 cases:

\begin{itemize}
\item $\varphi(uv_i)+1=\varphi(uv_j)+1$ for some $i<j\le t$. Then $\left(\varphi(v_iv_j), \varphi(v_iu), \varphi(v_ju)\right)\notin \ca$.

\item $\varphi(uv_i)+1=\varphi(uv_j)$ for some $i\le t<j$. Then $\left(\varphi(v_iu), \varphi(v_iv_j), \varphi(uv_j)\right)\notin \ca$.

\item $\varphi(uv_i)=\varphi(uv_j)$ for some $t<i<j$. Then $\left(\varphi(uv_i), \varphi(uv_j), \varphi(v_iv_j)\right)\notin \ca$.
\end{itemize}
This completes the proof of the lower bound.

 We then prove the upper bound. Indeed, suppose for the sake of contradiction that $\pa(S_k)>\frac{k^2-5k+7}{(k-1)^2}$. Then there exist palettes with density greater than $\frac{k^2-5k+7}{(k-1)^2}$ not admitted by $S_k$. Let $\cp=(\cc, \ca)$ be one such palette with the minimum number of colors. Removing any color from $\cc$ must decrease the density of $\cp$.   

Let $|\cc|=n$. Given colors $a,b\in\cc$, and $i\neq j\in\{1,2,3\}$, we say that $(a,b)$ is \emph{$(i,j)$-good} if there exists a triple $(c_1, c_2, c_3)\in\ca$ such that $c_i=a$ and $c_j=b$, otherwise $(a,b)$ is \emph{$(i,j)$-bad}.  Let $d_{i,j}(a)$ denote the number of colors $b$ such that $(a,b)$ is $(i,j)$-good, and $e_{i,j}(a)=d_{i,j}(a)/n$. Moreover, define $d_{i,j}'(a)=n-d_{i,j}(a)$, i.e, the number of colors $b$ such that $(a,b)$ is $(i,j)$-bad, and $e'_{i,j}(a)=d'_{i,j}(a)/n$. The following claim give an estimate on the lower bound of $d_{i,j}(a)$.

\begin{claim}\label{claim:degree}
    $e_{i,j}(a)\geq 3d(\cp)-2$ for all $a\in\cc$ and all $i\neq j\in\{1,2,3\}$.
\end{claim}

\begin{proof}
By symmetry, assume that $i=1,j=2$. Let $\cp'=(\cc', \ca')$ be the palette obtained by removing the color $a$ from $\cc$. Then
\begin{align*}
n^3d(\cp)=|\ca|
&=|\ca'|+|\ca\setminus \ca'|\\
&= |\ca'|+|\ca_a^{1}|+|\ca_a^{2}\setminus \ca_a^{1}|+|\ca_a^{3}\setminus(\ca_a^{1}\cup\ca_a^{2})|\\
&\le |\ca'|+nd_{i,j}(a)+n(n-1)+(n-1)^2\\
&\le (n-1)^3d(\cp)+nd_{i,j}(a)+2n^2-3n+1,
\end{align*}
where the last inequality uses the condition that removing any color from $\cc$ decreases the density of $\cp$. Hence,
\begin{align*}
nd_{i,j}(a)
&\ge (n^3-(n-1)^3)d(\cp)-2n^2+3n-1\\
&=d(\cp)(3n^2-3n+1)-2n^2+3n-1\\
&=(3d(\cp)-2)n^2+(3n-1)(1-d(\cp))\\
&\ge (3d(\cp)-2)n^2,
\end{align*}
which proves our claim.
\end{proof}

In the next lemma we bound the density of the palette by $e_{i,j}$.

\begin{lemma}\label{lem:incexc}
\[d(\cp)\le \frac{1}{4}+\frac{1}{2n}\sum_{a\in\cc}\sum_{i\neq j\in\{1,2,3\}}\left(e_{i,j}(a)-\frac {1}{2}\right)^2.\]
\end{lemma}

\begin{proof} Define
\begin{align*}
X_1=\{(a,b,c)|(a,b,c)\in \cc^3, \text{for all }d\in \cc, (d,b,c)\notin \ca\},\\
X_2=\{(a,b,c)|(a,b,c)\in \cc^3,\text{for all }d\in \cc, (a,d,c)\notin \ca\},\\
X_3=\{(a,b,c)|(a,b,c)\in \cc^3,\text{for all }d\in \cc, (a,b,d)\notin \ca\}.
\end{align*}

Clearly $\ca$ does not contain any element from $X_1$, $X_2$ or $X_3$, so by the inclusion-exclusion principle, we have 
\begin{align*}
|\ca|&\le |\overline{X_1}\cap\overline{X_2}\cap\overline{X_3}|\\
&\le n^3-|X_1|-|X_2|-|X_3|+|X_1\cap X_2|+|X_1\cap X_3|+|X_2\cap X_3|.
\end{align*}

 Note that $(a,b,c)\in X_1$ is equivalent to $(b,c)$ being $(2,3)$-bad, hence
\[|X_1|=n\sum_{a\in \cc}d'_{2,3}(a)=n\sum_{a\in \cc}d'_{3,2}(a)=\frac{n}{2}\left(\sum_{a\in \cc}d'_{2,3}(a)+\sum_{a\in \cc}d'_{3,2}(a)\right).\]

Besides, $(a,b,c)\in X_1\cap X_2$ is equivalent to $(b,c)$ being $(2,3)$-bad and $(a,c)$ being $(1,3)$-bad, hence
\[|X_1\cap X_2|=\sum_{a\in\cc}d'_{3,1}(a)d'_{3,2}(a).\]

Using the same argument for $X_2,X_3$, we have

\begin{align*}
|\ca|&\le n^3-|X_1|-|X_2|-|X_3|+|X_1\cap X_2|+|X_1\cap X_3|+|X_2\cap X_3|\\
&\le n^3-\frac{n}{2}\sum_{a\in\cc}\sum_{i\neq j}d'_{i,j}(a)+\sum_{a\in\cc}\left(d'_{1,2}(a)d'_{1,3}(a)+d'_{2,1}(a)d'_{2,3}(a)+d'_{3,1}(a)d'_{3,2}(a)\right),
\end{align*}
which is equivalent to
\[d(\cp)\le 1+\frac{1}{n}\sum_{a\in\cc}\left(e'_{1,2}(a)e'_{1,3}(a)+e'_{2,1}(a)e'_{2,3}(a)+e'_{3,1}(a)e'_{3,2}(a)-\frac{1}{2}\sum_{i\neq j}e'_{i,j}(a)\right).\]

Note that we have the local inequality
\[ab-\frac{1}{2}(a+b)=\left(a-\frac{1}{2}\right)\left(b-\frac{1}{2}\right)-\frac{1}{4}\le \frac{1}{2}\left(\big(a-\frac{1}{2}\big)^2+\big(b-\frac{1}{2}\big)^2\right)-\frac{1}{4},\]

Hence
\begin{align*}
d(\cp)
&\le 
1+\frac{1}{n}\sum_{a\in\cc}\left(e'_{1,2}(a)e'_{1,3}(a)+e'_{2,1}(a)e'_{2,3}(a)+e'_{3,1}(a)e'_{3,2}(a)-\frac{1}{2}\sum_{i\neq j}e'_{i,j}(a)\right)\\
&\le 1+\frac{1}{n}\sum_{a\in\cc}\sum_{i\neq j\in\{1,2,3\}}\frac12\left(e'_{i,j}(a)-\frac {1}{2}\right)^2-\frac{3}{4}\\
&=\frac{1}{4}+\frac{1}{2n}\sum_{a\in\cc}\sum_{i\neq j\in\{1,2,3\}}\left(e_{i,j}(a)-\frac {1}{2}\right)^2.
\end{align*}

\end{proof}

Now we want to use the condition that $\cp$ does not admit $S_k$ to bound $e_{i,j}(a)$. To give a better characterization of the condition, we define a digraph $D(V,E)$ as follows:
\begin{itemize}
\item $V=\cc_1\cup\cc_2$, which consists of two disjoint copies of $\cc$. Besides, for each color $a\in\cc$, let $a^1$ and $a^2$ denote the copies of $a$ in $\cc_1$ and $\cc_2$.

\item For a pair of colors $(a,b)\in \cc^2$, 
\begin{itemize}
\item we add an edge $a^1b^1$ in $\cc_1$ if $(a,b)$ is $(2,3)$-good,
\item we add an edge $a^2b^2$ in $\cc_2$ if $(a,b)$ is $(1,2)$-good,
\item we add the edges $a^1b^2$ and $b^2a^1$ if $(a,b)$ is $(1,3)$-good.
\end{itemize}
Here, we use $ab$ to express the arc from $a$ to $b$.

\end{itemize}

One important thing is to show that $D$ is well-defined, i.e, $D$ has no loop. Assume that $D$ contains an edge $a^1a^1$, then $(a,a)$ is $(2,3)$-good, hence there exists a triple $(b,a,a)\in \mathcal{A}$. Recall $S_k$ be the $3$-graph on $k+1$ vertices $u,v_1, v_2, \dots, v_k$, containing the edges $uv_iv_j$ for all $1\leq i<j\leq k$. We define a order $\preceq$ on $V(F)$ as $v_1\preceq\cdots\preceq v_k \preceq u$, and  construct a function 
$\varphi:\binom{V(F)}{2}\rightarrow \mathcal{C}$ such that $\varphi(v_iv_j)=b, \varphi(v_iu)=a$. Then, for every edge $v_iv_ju\in E(F)$ with 
$v_i\preceq v_j\preceq u$, we have $\left(\varphi(ab), \varphi(ac), \varphi(bc)\right)=(a,a,b)\in \mathcal{A}$, a contradiction. Similarly,  $D$ also cannot contains an edge $a^2a^2$, hence $D$ is well-defined.

\begin{claim}\label{cla:trans}
    $D$ does not contain a transitive tournament on $k$ vertices as a subgraph.
\end{claim}

\begin{proof}

We prove the claim by contradiction. Assume that $K$ is a transitive tournament on $k$ vertices in $D$, and $V(K)\cap \cc_1=K_1,V(K)\cap \cc_2=K_2$. Let $|K_1|=s, |K_2|=t$. We can label the vertices of this transitive tournament $\ell_1^1, \ell_2^1, \dots, \ell_s^1, r_1^2, r_2^2, \dots, r_t^2$ in such a way that we have the edge $\ell_i^{1}\ell_j^{1}$ for any $1\le i<j\le s$ and the edge $r_i^{2}r_{j}^{2}$ for any $1\le i<j\le t$. By the definition of $D$, we can find some $\ell_{ij}$ such that $(\ell_{i,j},\ell_i,\ell_j)\in \ca$, some $m_{ij}$ such that $(\ell_i,m_{ij},r_j)\in \ca$ and some $r_{ij}$ such that $(r_i,r_j,r_{ij})\in\ca$. 

Now we show that $\cp$ admits $S_k$, which is a contradiction. Note that $s+t=k$, to be convenient, assume the star $S_k$ has vertices $u,v_1,\cdots,v_{s},w_1,\cdots,w_t$. We define an order $\preceq$ on $V(S_k)$ as follows:

\[v_1\preceq\cdots\preceq v_s\preceq u \preceq w_1\preceq\cdots\preceq w_t.\]

We then construct  a function $\varphi:\binom{V(S_k)}{2}\rightarrow \mathcal{C}$
such that for every edge $abc\in E(F)$ with $a\prec b\prec c$
we have $\left(\varphi(ab), \varphi(ac), \varphi(bc)\right)\in \mathcal{A}$. Define
\[\varphi(uv_i)=\ell_i, \; \varphi(uw_i)=r_i, \; \varphi(v_iv_j)=\ell_{ij}, \; \varphi(v_iw_j)=m_{ij}, \; \varphi(w_iw_j)=r_{ij}.\]

By the definition of $S_k$, $u\in\{a,b,c\}$. Now we have 3 cases:

\begin{itemize}
\item $a=u,b=w_i, c=w_j$. Then $\left(\varphi(ab), \varphi(ac), \varphi(bc)\right)=(r_i,r_j,r_{ij})\in \ca$.

\item $b=u,a=\ell_i, c=w_j$. Then $\left(\varphi(ab), \varphi(ac), \varphi(bc)\right)=(\ell_i,m_{i,j},r_j)\in \ca$, 

\item $c=u,a=\ell_i, c=\ell_j$. Then $\left(\varphi(ab), \varphi(ac), \varphi(bc)\right)=(\ell_{i,j},\ell_i,\ell_j)\in \ca$.
\end{itemize}
\end{proof}

We then introduce a variation of the Caro-Wei theorem in digraph. It is similar to Theorem 4 in \cite{gruber2011bounding}.

\begin{lemma}\label{lem:inverse}
Let $H$ be a digraph on $n$ vertices, which does not contain a transitive tournament on $k$ vertices as a subgraph. For each vertex $v$, let $m(v)=\max\{d^+(v), d^-(v)\}$. Then \[\sum_{v\in V(H)}\frac{1}{n-m(v)}\leq k-1.\]
\end{lemma}

\begin{proof}

We proceed by induction on $k$. The statement is clear for $k=2$, since $H$ is the empty graph. Suppose that the statement holds for $k-1$. Let $w$ be a vertex with the maximum value of $m(w)$, and by symmetry assume that $m(w)=d^+(w)$. Let $S$ be the set of out-neighbors of $w$, and $D[S]$ be the digraph induced by $D$ in $S$. Then $D[S]$ does not contain a transitive tournament on $k-1$ vertices, otherwise we could extend it by adding $w$. Let $d^+_S(v)$ and $d^-_S(v)$ denote the out-degree and in-degree of $v$ in $G_S$, and let $m_S(v)=\max\{d^+_S(v), d^-_S(v)\}$. Hence,

\begin{align*}
\sum_{v\in v(H)}\frac{1}{n-m(v)}&=\sum_{v\in S}\frac{1}{n-m(v)}+\sum_{v\in V(H)\setminus S}\frac{1}{n-m(v)}\\
&\le \sum_{v\in S}\frac{1}{n-m(v)}+\sum_{v\in V(H)\setminus S}\frac{1}{n-m(w)} && (\text{by the maximality of }m(w))\\
&\le \sum_{v\in S}\frac{1}{n-m(v)}+1\\
&\le \sum_{v\in S}\frac{1}{|S|-m_S(v)}+1  &&(m(v)\leq m_S(v)+(n-|S|))\\
&\le k-1. && (\text{by the induction hypothesis})
\end{align*}

\end{proof}

We then combine \cref {cla:trans} and \cref{lem:inverse} to give an estimation of $e_{i,j}$. 

\vskip 0.1 em

For each $a\in\cc$, set 
\begin{align*}
m_A(a)&=\max\{e_{2,3}(a), e_{3,2}(a)\},\quad \quad m_C(a)=\max\{e_{1,2}(a), e_{2,1}(a)\}, \\
m_B(a)&=m_A(a)+e_{1,3}(a),\quad \quad \quad \quad \, m_D(a)=m_C(a)+e_{3,1}(a).
\end{align*}

For each $a\in \ca$, set $M_A(a)=\frac{1}{1-m_A(a)}$, $M_C(a)=\frac{1}{1-m_C(a)}$. Consider the digraph $D[\cc_1]$. For each $a\in\cc$, the maximum of its in-degree and out-degree in $D[\cc_1]$ is $m_A(a)n$, so by Lemma~\ref{lem:inverse}, we have 
\begin{align}\label{1}
\sum_{a\in\cc}M_A(a)=\sum_{a\in\cc}\frac{1}{1-m_A(a)}\leq (k-1)n,
\end{align}
and the inequality is also true when we replace $M_A(a)$ with $M_C(a)$. 

\vskip 0.15 em

Recall that $m(v)=\max\{d^+(v), d^-(v)\}$. In $D$, for each $a\in\cc$, $m(a^1)=m_B(a)n$, $m(a^2)=m_D(a)n$. Set $M_B(a)=\frac{1}{2-m_B(a)}, M_D(a)=\frac{1}{2-m_D(a)}$, then by Lemma~\ref{lem:inverse} again we have  
\begin{align}\label{2}
\sum_{a\in \cc}\big(M_B(a)+M_D(a)\big)=\sum_{a\in \cc}\left(\frac{1}{2-m_B(a)}+\frac{1}{2-m_D(a)}\right)\le  (k-1)n.
\end{align}

Now, by Lemma~\ref{lem:incexc}, we have

\begin{align}\label{3}
d(\cp)
&\le \frac{1}{4}+\frac{1}{2n}\sum_{a\in\cc}\sum_{i\neq j\in\{1,2,3\}}\left(e_{i,j}(a)-\frac {1}{2}\right)^2 \nonumber\\
&\le \frac14+\frac{1}{2n}\sum\limits_{a\in\cc}\left(2\left(m_A(a)-\frac 12\right)^2+\left(m_B(a)-m_A(a)-\frac12\right)^2\right)\nonumber\\
&+\frac{1}{2n}\sum\limits_{a\in\cc}\left(2\left(m_C(a)-\frac 12\right)^2+\left(m_D(a)-m_C(a)-\frac12\right)^2\right). 
\end{align}

Note that for all real numbers $a,b$ we have the local inequality
\begin{align}\label{4}
2(a-\frac{1}{2})^2+(b-a-\frac{1}{2})^2
&=(a-\frac{1}{2})^2+\frac{1}{2}(b-1)^2+2(a-\frac{b}{2})^2\nonumber\\
&\le (a-\frac{1}{2})^2+\frac{1}{2}(b-1)^2+4(\frac{k-2}{k-1}-a)^2+4(\frac{k-2}{k-1}-\frac{b}{2})^2\nonumber\\
&=f(a)+g(b),
\end{align}
where $f(x)=(x-\frac{1}{2})^2+4(\frac{k-2}{k-1}-x)^2$, $g(x)=\frac{1}{2}(x-1)^2+4(\frac{k-2}{k-1}-\frac{x}{2})^2$.

\vskip 0.4 em

Substituting \eqref{4} in \eqref{3}, we have
\begin{align}\label{5}
d(\cp)
&\le \frac{1}{4}+\frac{1}{2n}\sum_{a\in\cc}\big( f(m_A(a))+f(m_C(a))+g(m_B(a))+g(m_D(a)) \big)\nonumber\\
&= \frac{1}{4}+\frac{1}{2n}\sum_{a\in\cc}\big( f_1(M_A(a))+f_1(M_C(a))+g_1(M_B(a))+g_1(M_D(a)) \big),
\end{align}
where  $f_1(x)=f(1-\frac{1}{x})$, $g_1(x)=g(2-\frac{1}{x})$. We can calculate that
\[f_1(x)=\frac{5}{x^2}-\frac{(k+7)}{(k-1)x}+\frac{4}{(k-1)^2}+\frac{1}{4}, \quad  g_1(x)=\frac{3}{2x^2}-\frac{(k+3)}{(k-1)x}+\frac{4}{(k-1)^2}+\frac{1}{2}.\]

Now we bound the four sums in \eqref{5} separately.  The idea is to use a tangent-type local inequality.

\begin{claim}\label{cla:L1}
For any $x\ge \frac{5(k-1)}{k-3}$,
\[f_1(x)\le \frac{(k-3)}{(k-1)^3}x+\frac{k^2-10k+21}{(2k-2)^2}.\]
\end{claim}

\begin{proof}
The inequality is equivalent to $(x-(k-1))^2((k-3)x-5(k-1))\ge 0$.
\end{proof}

By Claim~\ref{claim:degree}, we have $m_A(a)\ge 3d(\cp)-2$ for every $a\in \ca$, therefore 
\[M_A(a)\ge \frac{1}{3-3d(\cp)}\ge \frac{(k-1)^2}{9(k-2)}\ge \frac{5(k-1)}{k-3}\]

when $k\ge 48$. Hence, by Claim~\ref{cla:L1} and \eqref{1}, we have
\begin{align}\label{6}
\sum_{a\in\cc}f_1(M_A(a))
&\le \sum_{a\in\cc}\left(\frac{(k-3)}{(k-1)^3}M_A(a)+\frac{k^2-10k+21}{(2k-2)^2}\right)\nonumber\\
&\le \frac{(k-3)}{(k-1)^3}(k-1)n+\frac{k^2-10k+21}{(2k-2)^2}n\nonumber\\
&=n\frac{(k-3)^2}{(2k-2)^2}.
\end{align}

Similarly, we have
\begin{align}\label{7}
\sum_{a\in\cc}f(m_C(a)) \le n\frac{(k-3)^2}{(2k-2)^2}. 
\end{align}

\begin{claim}\label{cla:L2}
For any $x\ge \frac{3(k-1)}{2k-6}$,
\[g_1(x)\le \frac{4k-12}{(k-1)^3}x+\frac{k^2-10k+21}{2(k-1)^2}. \]
\end{claim}

\begin{proof}
The inequality is equivalent to $(2x-(k-1))^2((2k-6)x-3(k-1))\ge 0$.
\end{proof}

By Claim~\ref{claim:degree}, we have $m_B(a)\ge 6d(\cp)-4$ for every $a\in \ca$, then 

\[M_B(a)\ge \frac{1}{6-6d(\cp)}\ge \frac{(k-1)^2}{18(k-2)}\ge \frac{3(k-1)}{2k-6}\]
when $k\ge 30$. Similarly, $M_D(a)\ge \frac{3(k-1)}{2k-6}$, hence by Claim~\ref{cla:L2} and  \eqref{2} we have
\begin{align}\label{8}
\sum_{a\in\cc}\big(g_1(M_B(a))+g_1(M_D(a)\big)&\le \sum_{a\in\cc}\left(\frac{4k-12}{(k-1)^3}(M_B(a)+M_D(a))+\frac{k^2-10k+21}{(k-1)^2}\right)\nonumber\\
&\le \frac{4k-12}{(k-1)^3}(k-1)n+\frac{k^2-10k+21}{(k-1)^2}n\nonumber\\
&=n\frac{(k-3)^2}{(k-1)^2}. \tag{$8$}
\end{align}

Hence, putting together \eqref{5}, \eqref{6}, \eqref{7} and \eqref{8}, we have
\begin{align*}
d(\cp)&\le \frac{1}{4}+\frac{1}{2n}\sum_{a\in\cc}\big( f(m_A(a))+f(m_C(a))+g(m_B(a))+g(m_D(a)) \big)\\
&\le \frac{1}{4}+\frac{3(k-3)^2}{4(k-1)^2}\\
&=\frac{k^2-5k+7}{(k-1)^2}.
\end{align*}

which finishes the proof.
\end{section}

\begin{section}{Concluding remarks}

Using computer calculations, one can show that the inequalities \eqref{6}, \eqref{7} hold for $k\ge 40$, which means that \cref{thm:main} is also true for $k\ge 40$. 

\end{section}

\paragraph{Acknowledgements}

The authors want to thank Haoran Luo for computer calculation. The authors are also grateful to  Daniel Kr\'al, Hong Liu and Oleg Pikhurko for careful reading and some writing suggestions.

\bibliographystyle{abbrv}
\bibliography{bibliog.bib}
\end{document}